\documentclass[11pt]{amsart}
\usepackage[T1]{fontenc}
\usepackage{lmodern}
\usepackage{amsmath,amssymb,amsthm}
\usepackage[margin=1in]{geometry}
\usepackage{microtype}
\usepackage[hidelinks]{hyperref}
\usepackage{graphicx}

\newtheorem{theorem}{Theorem}
\newtheorem{lemma}[theorem]{Lemma}
\newtheorem{proposition}[theorem]{Proposition}
\newtheorem{corollary}[theorem]{Corollary}

\theoremstyle{definition}
\newtheorem{remark}[theorem]{Remark}

\DeclareMathOperator{\GL}{GL}
\DeclareMathOperator{\SL}{SL}
\newcommand{\Q}{\mathbb Q}
\newcommand{\Z}{\mathbb Z}
\newcommand{\R}{\mathbb R}
\newcommand{\F}{\mathbb F}

\title{Linearity of tubular groups}

\author[D.~T.~Wise]{Daniel T. Wise}
\email{daniel.wise@weizmann.ac.il}
\address{Dept. of Mathematics\\
Weizmann Institute of Science}

\date{\today}
\hypersetup{pdftitle={Linearity of tubular groups}}
\begin{document}
\begin{abstract}
We prove that a tubular group is linear if and only if it is residually
finite, equivalently virtually primitive. Every such group embeds in
$\SL_n(\Z[u])$. For a primitive tubular decomposition with at most $m$
edge slopes at each vertex, our representation has dimension
$2\max\{2,m\}$. In particular, Gersten's group embeds in $\SL_6(\Z[u])$,
and this dimension is minimal.
\end{abstract}
\maketitle

A \emph{tubular group} is a group that splits as a finite
graph of groups with vertex groups $\Z^2$ and edge groups $\Z$.
It is \emph{primitive} if each edge group maps onto
a maximal cyclic subgroup of each adjacent vertex group.
Thus its edge generators are primitive vectors $(p,q)$, with
$\gcd(p,q)=1$. It is \emph{virtually primitive} if it has a
finite-index subgroup whose induced splitting is primitive.
It is \emph{linear} if it embeds in $\GL_n(\F)$ for some integral domain $\F$
and some $n$. We then say it is \emph{$\F$-linear}.

\begin{theorem}\label{thm:main}
For a tubular group $G$, the following are equivalent:
\begin{enumerate}
\item $G$ is residually finite;
\item $G$ is virtually primitive;
\item $G$ is linear;
\item $G$ embeds in $\SL_n(\Z[u])$ for some $n$, where $u$ is an
indeterminate.
\end{enumerate}
In particular, every residually finite tubular group is $\R$-linear.
\end{theorem}

We assign a two-dimensional summand to each edge slope, so that all
edge generators have the same unipotent Jordan form. A polynomial shear
then gives edge matrices. We choose the changes of basis so that a single
matrix entry vanishes exactly at a pinch. The following lemma shows that
each transition without a pinch increases the degree by one.

We give the integral refinement in Section~\ref{sec:integral}, and
illustrate degree growth for Gersten's group in Section~\ref{sec:gersten}.

This is the leading-coefficient method of
Shalen \cite[Prop.~1.3]{Shalen}. Closely related matrix-entry
combination criteria appear in Douba--Tsouvalas
\cite[Thm.~2.1]{DTentries} and \cite[Thm.~6]{DTflags}.

\section{The rational construction}\label{sec:rational}

The \emph{degree} of a polynomial matrix is the largest degree of an entry.
For a polynomial matrix $A(u)$, let $[u^k]A(u)$ denote the matrix
obtained by taking the coefficient of $u^k$ in each entry.

\begin{lemma}[Degree growth]\label{lem:degree}
Let $\F$ be an integral domain and let $N\ge2$.
Let $E\in M_N(\F)$ be the matrix with a single $1$ in position
$(1,2)$ and zeros elsewhere. Let $P(u)=I+uE$.
Suppose $W(u)\in M_N(\F[u])$ has degree $n$ and leading
coefficient $cE$, where $c\ne0$.
If $M\in M_N(\F)$ satisfies $M_{21}\ne0$ and $\varepsilon=\pm1$,
then $W(u)MP(\varepsilon u)$ has degree $n+1$ and leading coefficient
$\varepsilon cM_{21}E$.
\end{lemma}

\begin{proof}
The degree is at most $n+1$, and $EME=M_{21}E$ gives
\[
[u^{n+1}]\bigl(W(u)MP(\varepsilon u)\bigr)
=\varepsilon cEME=\varepsilon cM_{21}E\ne0.
\qedhere
\]
\end{proof}

The leading coefficient remains a nonzero multiple of $E$, so the lemma
can be iterated.

\begin{proposition}\label{prop:rational}
Let $G$ admit a primitive tubular decomposition, and let $m$ be the
maximum number of distinct incident edge slopes at a vertex. Then
$G\hookrightarrow\SL_{2d}(\Q[u])$, where $d=\max\{2,m\}$.
\end{proposition}

\begin{proof}
Let $H_v$ be the vertex group at $v$, and choose a generator of each
edge group. For an oriented edge $e:v\to w$,
let $c_e\in H_v$ and $c_{\bar e}\in H_w$ be its two images, where
$\bar e$ denotes the reverse edge. Reversing the edge exchanges these
two attaching elements. We use based path words, with the vertex-group
relations and the path relations
\[
e\bar e=1,\ c_e e=e c_{\bar e}.
\]
A \emph{pinch} is a subword $e h\bar e$ with
$h\in\langle c_{\bar e}\rangle$; these relations remove its two edges.

Let $V=\Q^{2d}$ with basis $p_1,q_1,\ldots,p_d,q_d$.
Fix a vertex $v$ and identify $H_v$ with $\Z^2$.
Choose primitive vectors $w_1,\ldots,w_d\in H_v$ of distinct
unoriented slopes, including all incident edge slopes, and let
$\lambda_j(h)=\det(w_j,h)$. Let
\[
\rho_v(h)p_j=p_j+\lambda_j(h)q_j,\qquad
\rho_v(h)q_j=q_j.
\]
Additivity makes $\rho_v$ a homomorphism, and two distinct slopes
give independent functionals, so it is faithful.

For each oriented edge, we now choose a basis with two purposes:
to put its attaching element into a common Jordan form, and to make
the $(2,1)$-entry of each transition matrix detect a pinch.
Choose rational numbers $r_e\in[0,1)$, distinct for distinct oriented
edges of the same slope starting at $v$.
Let $j(e)$ be the slope index of $c_e$.
For $j=j(e)$, let $S_e$ have columns
\[
p_j+r_e q_j,\quad q_j,\quad
p_\ell+q_j,\quad \lambda_\ell(c_e)q_\ell
\qquad(\ell\ne j),
\]
listing the last pairs in increasing order of $\ell$.
They form a basis since $\lambda_\ell(c_e)\ne0$ for $\ell\ne j$.
The first two columns are fixed by $\rho_v(c_e)$, and in every remaining
pair the second is added to the first. Thus
\begin{equation}\label{eq:normal}
\rho_v(c_e)S_e=S_eJ,\qquad
J=I_2\oplus J_2^{\oplus(d-1)},\qquad
J_2=\begin{pmatrix}1&0\\1&1\end{pmatrix}.
\end{equation}
For edges $e,f$ starting at $v$ and $h\in H_v$, the entry
$(S_e^{-1}\rho_v(h)S_f)_{21}$ is the second coordinate, in the
$S_e$-basis, of the image under $\rho_v(h)$ of the first column of $S_f$.
If $j(f)=\ell\ne j$, the equality $p_\ell=(p_\ell+q_j)-q_j$ gives
this coordinate as $-1$. If $j(f)=j$, subtracting
$p_j+r_e q_j$ leaves $(\lambda_j(h)+r_f-r_e)q_j$. Thus
\[
\bigl(S_e^{-1}\rho_v(h)S_f\bigr)_{21}=
\begin{cases}
-1,&j(e)\ne j(f),\\[2pt]
\lambda_j(h)+r_f-r_e,&j(e)=j(f)=j.
\end{cases}
\]
Each $\lambda_j(h)$ is integral, whereas distinct parameters differ by a
nonzero number in $(-1,1)$. Primitivity gives
$\ker\lambda_{j(e)}=\langle c_e\rangle$, so
\begin{equation}\label{eq:detection}
\bigl(S_e^{-1}\rho_v(h)S_f\bigr)_{21}=0
\quad\Longleftrightarrow\quad
e=f\ \text{ and }\ h\in\langle c_e\rangle.
\end{equation}

Repeat this construction at every vertex, with the same $V$ and $J$.
Multiply the first column of each $S_e$ by $(\det S_e)^{-1}$.
This preserves \eqref{eq:normal} and multiplies each transition entry
by a nonzero scalar, so \eqref{eq:detection} still holds.
Thus we may assume that every $S_e$ has determinant one.
Let $E$ and $P(u)=I+uE$ be as in Lemma~\ref{lem:degree}.
Choose signs $\varepsilon_e=\pm1$ with
$\varepsilon_{\bar e}=-\varepsilon_e$, and let
\begin{equation}\label{eq:edge}
X_e=S_eP(\varepsilon_e u)S_{\bar e}^{-1}.
\end{equation}
The shear $P(u)$ acts within the identity block of $J$, so it commutes
with $J$. Since $P(u)^{-1}=P(-u)$, we have
\[
X_{\bar e}=X_e^{-1},\qquad
\rho_v(c_e)X_e=X_e\rho_w(c_{\bar e})\quad(e:v\to w).
\]
Every vertex matrix and every edge matrix has determinant one.
Thus products along loops at a fixed base vertex define
$\rho:G\to\SL_{2d}(\Q[u])$.

Let $g\ne1$ and choose a based path word with the fewest edge traversals:
\[
g=h_0e_1h_1\cdots e_nh_n,\qquad
e_i:v_{i-1}\to v_i,\quad h_i\in H_{v_i},\quad v_n=v_0.
\]
If $n=0$, vertex-group faithfulness gives $\rho(g)\ne I$.
Suppose $n>0$. Minimality excludes pinches, so \eqref{eq:detection} gives
\[
M_i:=S_{\bar e_i}^{-1}\rho_{v_i}(h_i)S_{e_{i+1}}
\quad\text{satisfies}\quad (M_i)_{21}\ne0
\qquad(1\le i<n).
\]
Expanding the edge matrices gives
\[
\rho(g)=A\bigl(P(\varepsilon_{e_1}u)M_1P(\varepsilon_{e_2}u)\cdots
M_{n-1}P(\varepsilon_{e_n}u)\bigr)B,
\]
where $A=\rho_{v_0}(h_0)S_{e_1}$ and
$B=S_{\bar e_n}^{-1}\rho_{v_0}(h_n)$ are invertible constant matrices.
The first shear has degree one and leading coefficient $\pm E$.
Lemma~\ref{lem:degree} shows that each subsequent transition increases
the degree by one. Invertible constant factors preserve degree, so
$\deg\rho(g)=n>0$. Thus $\rho(g)\ne I$.
\end{proof}

We exhibit Proposition~\ref{prop:rational} on 
Gersten's group in Section~\ref{sec:gersten}.

\section{The integral refinement}\label{sec:integral}

We realize the same vertex representations on integral lattices.
The following lemma supplies the required changes of basis.

\begin{lemma}[Integral bases]\label{lem:frames}
Let $V$ be a free abelian group of finite rank and let
$N\in\operatorname{End}(V)$ satisfy $N^2=0$.
Suppose $x\in\ker N$ is nonzero and $\varphi:V\to\Z$ is a homomorphism satisfying
\[
\varphi N=0,\qquad \varphi(\ker N)=\Z,\qquad
\ker N\cap\ker\varphi=\Z x\oplus NV.
\]
Let $r=\operatorname{rank}NV$. Then there is a basis
\[
x,z,y_1,Ny_1,\ldots,y_r,Ny_r
\]
of $V$ in which $\varphi$ is the second coordinate function and
\[
I+N=J:=I_2\oplus J_2^{\oplus r},\qquad
J_2=\begin{pmatrix}1&0\\1&1\end{pmatrix}.
\]
\end{lemma}

\begin{proof}
Choose $z\in\ker N$ with $\varphi(z)=1$. Then
$\ker N=\Z z\oplus\Z x\oplus NV$.
Since $V/\ker N\cong NV$ is free, lift a basis of this quotient to
$y_1,\ldots,y_r$. The vectors $Ny_i$ form a basis of $NV$.
Replacing $y_i$ by $y_i-\varphi(y_i)z$ makes $\varphi(y_i)=0$.
The displayed basis has the required properties.
\end{proof}

\begin{proposition}\label{prop:integral}
Let $G$ admit a primitive tubular decomposition, and let $m$ be the
maximum number of distinct incident edge slopes at a vertex. Then
$G\hookrightarrow\SL_{2d}(\Z[u])$, where $d=\max\{2,m\}$.
\end{proposition}

\begin{proof}
Let $H_v$ and $c_e$ be as in the proof of Proposition~\ref{prop:rational}.
The edge-matrix construction and degree argument there apply over any
integral domain. It therefore suffices to give faithful vertex
representations over $\Z$ and integral bases $S_e$ satisfying
\eqref{eq:normal} and \eqref{eq:detection}.

Let $R_k$ be the integer span of $\xi^k,\xi^{k-1}\eta,\ldots,\eta^k$.
For $h=(a,b)$, let $L_h=b\xi-a\eta$, so
$L_h(w)=\det(w,h)$. For a primitive vector $c$ and $k\ge1$, evaluation
at $c$ gives
\begin{equation}\label{eq:eval}
R_k\twoheadrightarrow\Z,\qquad F\longmapsto F(c),\qquad
\ker(\operatorname{ev}_c)=L_cR_{k-1}.
\end{equation}
Indeed, complete $c$ to an integral basis and use its coordinates
$(\xi,\eta)$. Then $L_c=\pm\eta$, and evaluation selects the coefficient
of $\xi^k$.

Fix a vertex $v$ and identify $H_v$ with $\Z^2$.
Choose primitive vectors $w_1,\ldots,w_d\in H_v$ of distinct
unoriented slopes, including all incident edge slopes. Let
\[
L_j=L_{w_j},\qquad \Pi=L_1\cdots L_d,\qquad
\Pi_j=\Pi/L_j,\qquad \pi_j=\Pi_j(w_j)\ne0,
\]
and let $\lambda_j(h)=L_h(w_j)=\det(w_j,h)$.
The polynomial $\Pi_j$ vanishes at every $w_i$ except $w_j$.
Successive division using \eqref{eq:eval} shows that a polynomial in
$R_d$ vanishing at all the $w_i$ is an integer multiple of $\Pi$;
in $R_{d-1}$ it is zero. Thus evaluation embeds both $R_{d-1}$ and
$R_d/\Z\Pi$ in $\Z^d$. Let
\[
V_v=R_{d-1}\oplus(R_d/\Z\Pi),
\]
which is free abelian of rank $2d$, and let
\[
\rho_v(h)(F,Q)=(F,Q+L_hF).
\]
Over $\Q$, evaluation gives a two-dimensional summand for each slope,
on which $\rho_v(h)$ is the shear
\[
(F(w_j),Q(w_j))\longmapsto
(F(w_j),Q(w_j)+\lambda_j(h)F(w_j)).
\]
Additivity makes $\rho_v$ a unipotent representation, and two distinct
slopes give independent functionals, so it is faithful.

We now choose bases that put the attaching elements into a common
normal form and detect pinches.

For an oriented edge $e$ starting at $v$, let $j=j(e)$ be the slope
index of $c_e$, so $c_e=\pm w_j$. Let $N_e=\rho_v(c_e)-I$.
Since $N_e(F,Q)=(0,\pm L_jF)$, cancellation and \eqref{eq:eval} give
\[
\ker N_e=\Z\Pi_j\oplus(R_d/\Z\Pi),\qquad
N_eV_v=\{(0,Q):Q(w_j)=0\}.
\]
Indeed, $L_jF\in\Z\Pi$ gives $F\in\Z\Pi_j$, while $Q(w_j)=0$
gives a representative $L_jH$ with $H\in R_{d-1}$.
Choose positive integers $t_e$, prime to $\pi_{j(e)}$, which are distinct
for distinct oriented edges of the same slope starting at $v$.
For each slope $j$, distinct positive integers congruent to $1$ modulo
$|\pi_j|$ suffice. Choose $K_j\in R_1$ with $K_j(w_j)=1$, and let
\[
x_e=(t_e\Pi_j,K_j\Pi_j),\qquad
\varphi_e(F,Q)=t_eQ(w_j)-\sum_i F(w_i).
\]
We have $N_ex_e=0$ and $\varphi_eN_e=0$. On the kernel,
\[
\varphi_e(a\Pi_j,Q)=t_eQ(w_j)-a\pi_j.
\]
Since $\gcd(t_e,\pi_j)=1$ and $Q(w_j)$ takes every integer value,
$\varphi_e(\ker N_e)=\Z$. Its value is zero precisely when
$a=st_e$ and $Q(w_j)=s\pi_j$ for some $s\in\Z$. Hence
\[
\ker N_e\cap\ker\varphi_e=\Z x_e\oplus N_eV_v.
\]
By Lemma~\ref{lem:frames}, there is a basis matrix
$S_e:\Z^{2d}\to V_v$ whose first column is $x_e$ and whose inverse
has second row $\varphi_e$, such that
\[
\rho_v(c_e)S_e=S_eJ,\qquad J=I_2\oplus J_2^{\oplus(d-1)}.
\]

For oriented edges $e,f$ starting at $v$ and $h\in H_v$, the transition
entry $(S_e^{-1}\rho_v(h)S_f)_{21}$ is
$\varphi_e(\rho_v(h)x_f)$. Let $j=j(e)$ and $i=j(f)$. We have
\[
\rho_v(h)x_f=(t_f\Pi_i,(K_i+t_fL_h)\Pi_i),
\]
so direct evaluation gives
\[
\bigl(S_e^{-1}\rho_v(h)S_f\bigr)_{21}=
\begin{cases}
-t_f\pi_i,&j\ne i,\\[2pt]
\pi_j\bigl(t_et_f\lambda_j(h)+t_e-t_f\bigr),&j=i.
\end{cases}
\]
Since $|t_e-t_f|<t_et_f$ and $\lambda_j(h)$ is integral, the second
expression vanishes precisely when $t_e=t_f$ and $\lambda_j(h)=0$.
Primitivity gives $\ker\lambda_{j(e)}=\langle c_e\rangle$, so
\eqref{eq:detection} holds.

Repeat this construction at every vertex and fix an integral basis
of each $V_v$. If $\det S_e=-1$, reverse the sign of its first column.
This preserves $\rho_v(c_e)S_e=S_eJ$ and changes only signs of the
transition entries, so \eqref{eq:detection} still holds.
Thus we may assume that every $S_e$ has determinant one.

With these bases, let $X_e$ be given by \eqref{eq:edge}.
The vertex and edge matrices lie in $\SL_{2d}(\Z[u])$.
The edge relations and the reduced-word argument in
Proposition~\ref{prop:rational} apply unchanged: a reduced based path
with $n>0$ edges has image of degree $n$, and vertex-group faithfulness
handles $n=0$. Thus the representation is faithful.
\end{proof}

\begin{proof}[Proof of Theorem~\ref{thm:main}]
The implication $(1)\Rightarrow(2)$ is due to Hoda--Wise--Woodhouse
\cite[Thm.~1.1]{HWW}.
For $(2)\Rightarrow(4)$, choose a finite-index subgroup $G_0$ with a
primitive tubular decomposition. Proposition~\ref{prop:integral}
gives a faithful representation $G_0\to\SL_N(\Z[u])$, with $N$ even.
Induction on the cosets $G/G_0$ gives a representation of $G$ over
$\Z[u]$ of dimension $N[G:G_0]$. It is faithful: an element in its
kernel fixes every coset, hence belongs to $G_0$, and acts on the
summand indexed by $G_0$ through the original faithful representation.
Its matrices permute the $N$-dimensional summands, with a determinant-one
block on each summand. The permutation contributes
$(\operatorname{sgn}\sigma)^N=1$, so the induced representation lies in
$\SL_{N[G:G_0]}(\Z[u])$.

The implication $(4)\Rightarrow(3)$ follows by passing to $\Q(u)$,
and $(3)\Rightarrow(1)$ is Mal\'cev's theorem \cite{Malcev} that
f.g.\ linear groups are residually finite.
Evaluating $u$ at a real number transcendental over $\Q$
embeds $\Z[u]$ in $\R$ and proves $\R$-linearity.
\end{proof}

\section{Gersten's group}\label{sec:gersten}
Gersten's group
\cite{Gersten} has the following presentation:
\begin{equation}\label{eq:gersten}
\langle a,b,x,y\mid
[a,b]=1,\ x^{-1}ax=ab,\ y^{-1}ay=ab^{-1}\rangle.
\end{equation}

\begin{corollary}
Gersten's group \eqref{eq:gersten}
embeds in $\SL_6(\Q[u])$ and in fact in $\SL_6(\Z[u])$.
\end{corollary}
\begin{proof}
The three edge slopes are $(1,0),(1,1),(1,-1)$. Then Propositions~\ref{prop:rational}~and~\ref{prop:integral} apply.
\end{proof}

\begin{remark}
Button proved that Gersten's group is not linear in positive
characteristic \cite[Cor.~4.7]{Button}, and has no faithful complex
representation of dimension at most four \cite[Thm.~7.1]{ButtonMin}.
Soroko extended the latter bound to five, as reported in
\cite[p.~148]{Button}. Thus our six-dimensional representation is
optimal. We give a proof of Soroko's result in
Proposition~\ref{prop:soroko}.
\end{remark}

Its three edge slopes are $(1,0),(1,1),(1,-1)$, so
Proposition~\ref{prop:rational} gives an embedding in $\SL_6(\Q[u])$.
We illustrate the construction the splitting of $G$ as a double HNN extension.
Let the basis be $(p_1,q_1,p_2,q_2,p_3,q_3)$, and let
\[
(\lambda_1,\lambda_2,\lambda_3)(a^m b^n)=(n,m-n,m+n).
\]
Only $x$ and $y$ have the same slope at their initial ends, so we may take
\[
(r_x,r_{x^{-1}},r_y,r_{y^{-1}})=(0,0,\tfrac12,0).
\]
With positive shear signs for $x,y$ and the determinant normalization
in the proof, the construction gives the following matrices
$A=\rho(a)$, $B=\rho(b)$, $X=\rho(x)$, and $Y=\rho(y)$.
\[
\setlength{\arraycolsep}{4pt}
A=\begin{pmatrix}
1&0&0&0&0&0\\
0&1&0&0&0&0\\
0&0&1&0&0&0\\
0&0&1&1&0&0\\
0&0&0&0&1&0\\
0&0&0&0&1&1
\end{pmatrix},\qquad
B=\begin{pmatrix}
1&0&0&0&0&0\\
1&1&0&0&0&0\\
0&0&1&0&0&0\\
0&0&-1&1&0&0\\
0&0&0&0&1&0\\
0&0&0&0&1&1
\end{pmatrix}.
\]
\[
\setlength{\arraycolsep}{3.5pt}
X=\begin{pmatrix}
-u&0&2&u&-u&0\\
0&0&0&1&0&0\\
1&0&0&0&0&0\\
0&1&0&0&0&0\\
0&0&0&0&1&0\\
0&0&0&0&0&\tfrac12
\end{pmatrix},\qquad
Y=\begin{pmatrix}
-u&0&-u&0&-2&u\\
-\frac{u}{2}&0&-\frac{u}{2}&0&-1&1+\frac{u}{2}\\
1&0&0&0&0&0\\
0&-1&0&0&0&0\\
0&0&1&0&0&0\\
0&0&0&\tfrac12&0&0
\end{pmatrix}.
\]

We illustrate the degree-growth argument using the integral representation
of Gersten's group. Identify
$\langle a,b\rangle$ with $\Z^2$ so that $a=(1,0)$ and $b=(0,1)$, with
coordinates $(\xi,\eta)$, and let $w_1,w_2,w_3$ be
$(1,0),(1,1),(1,-1)$. Then
\[
L_{a^mb^n}=n\xi-m\eta,\qquad
\Pi=\eta(\xi^2-\eta^2),\qquad
(\pi_1,\pi_2,\pi_3)=(-1,2,2).
\]
Only $x$ and $y$ have the same slope at their initial ends, so we may take
\[
(t_x,t_{x^{-1}},t_y,t_{y^{-1}})=(1,1,2,1).
\]
Let $K_j=\xi$ for each $j$. Use the basis $\xi^2,\xi\eta,\eta^2$ of $R_2$ followed by the basis
$\xi^3,\xi^2\eta,\xi\eta^2$ of $R_3/\Z\Pi$, in which $\eta^3=\xi^2\eta$.
With positive shear signs for $x,y$, one choice of the matrices $S_e$ gives
the following matrices $A=\rho(a)$, $B=\rho(b)$, $X=\rho(x)$ and
$Y=\rho(y)$.
\[
\text{\small\setlength{\arraycolsep}{3.5pt}$\displaystyle
A=\begin{pmatrix}
1 & 0 & 0 & 0 & 0 & 0 \\[1.5pt]
0 & 1 & 0 & 0 & 0 & 0 \\[1.5pt]
0 & 0 & 1 & 0 & 0 & 0 \\[1.5pt]
0 & 0 & 0 & 1 & 0 & 0 \\[1.5pt]
-1 & 0 & -1 & 0 & 1 & 0 \\[1.5pt]
0 & -1 & 0 & 0 & 0 & 1
\end{pmatrix},\qquad
B=\begin{pmatrix}
1 & 0 & 0 & 0 & 0 & 0 \\[1.5pt]
0 & 1 & 0 & 0 & 0 & 0 \\[1.5pt]
0 & 0 & 1 & 0 & 0 & 0 \\[1.5pt]
1 & 0 & 0 & 1 & 0 & 0 \\[1.5pt]
0 & 1 & 0 & 0 & 1 & 0 \\[1.5pt]
0 & 0 & 1 & 0 & 0 & 1
\end{pmatrix},$}
\]
\[
\text{\footnotesize\setlength{\arraycolsep}{2.6pt}$\displaystyle
X=\begin{pmatrix}
1-3u & 1 & -2u & -1+u & -1+u & -1+u \\[1.5pt]
0 & 1 & -1 & 0 & 0 & 0 \\[1.5pt]
3u & -2 & 1+2u & 1-u & 1-u & 1-u \\[1.5pt]
-3u & -1 & -2u & u & u & u \\[1.5pt]
0 & 1 & 0 & -1 & 0 & -1 \\[1.5pt]
3u & 0 & 2u & -u & -u & 1-u
\end{pmatrix},\quad
Y=\begin{pmatrix}
1-6u & 4 & 4-4u & -1+2u & 1-2u & -1+2u \\[1.5pt]
0 & 1 & 1 & 0 & 0 & 0 \\[1.5pt]
6u & -5 & -5+4u & 1-2u & -1+2u & 1-2u \\[1.5pt]
-3u & 1 & -2u & u & -u & u \\[1.5pt]
0 & -1 & 0 & 1 & 0 & -1 \\[1.5pt]
3u & 0 & 2u & -u & u & 1-u
\end{pmatrix}.$}
\]
They have determinant one and satisfy
$AB=BA$, $X^{-1}AX=AB$ and $Y^{-1}AY=AB^{-1}$.
For the pinchless word $y^{-1}byx^{-1}$, the two transition matrices are
\[
M_1=S_y^{-1}BS_y,\qquad M_2=S_{y^{-1}}^{-1}S_{x^{-1}},
\qquad (M_1)_{21}=4,\quad (M_2)_{21}=2.
\]
The successive products have the following degrees, witnessed by their
$(2,3)$-entries:
\[
\begin{array}{c|cccc}
W&Y^{-1}&Y^{-1}B&Y^{-1}BY&Y^{-1}BYX^{-1}\\ \hline
\deg W&1&1&2&3\\[2pt]
\text{leading term of }W_{23}&2u&2u&8u^2&-16u^3
\end{array}
\]
In the bases used in the proof, the leading-term calculations are
\[
(-uE)M_1(uE)=-4u^2E,\qquad
(-4u^2E)M_2(-uE)=8u^3E.
\]
Thus each further edge traversal contributes one power of $u$ with a
nonzero coefficient. Multiplication by the constant end factors gives
the leading terms in the table.

\section{The dimension bound}\label{sec:soroko}

\begin{proposition}[Soroko]\label{prop:soroko}
Gersten's group has no faithful representation in
$\GL_d(\mathbb C)$ for $d\le5$.
\end{proposition}

\begin{proof}
This depends crucially on Button's quasi-unipotence theorem
\cite[Thm.~4.5]{Button}, and the argument also recovers his dimension-four bound.

Suppose that $G$ as presented in \eqref{eq:gersten}
has a faithful complex representation of dimension at most five.
Since $ab^{-1},a,ab$ are conjugate, Button's theorem implies
that the eigenvalues of the image of $b$ are roots of unity.
For every $N\ge1$, the endomorphism
\[
a\mapsto a^N,\qquad b\mapsto b^N,\qquad x\mapsto x,\qquad y\mapsto y
\]
is injective: it is injective on $\langle a,b\rangle$ and preserves
reduced HNN words, since $h^N\in C$ if and only if $h\in C$
for every $h\in\langle a,b\rangle$ and every primitive attaching
subgroup $C$. Precomposing with this endomorphism for a suitable
$N$, we may assume that the image of $b$ is unipotent.

Some irreducible composition factor $\sigma$ satisfies
$\sigma(b)\ne I$: otherwise the subgroup
$\langle b,xbx^{-1}\rangle$, which is nonabelian free by
Britton's lemma, would embed in the nilpotent group acting
trivially on all composition factors.
The semisimple part of $\sigma(a)$ commutes with $\sigma(G)$,
by the conjugacy relations and uniqueness of Jordan decomposition,
and is therefore scalar. Rescaling $\sigma(a)$ preserves the
relations and irreducibility, so we may assume that both
$\sigma(a)$ and $\sigma(b)$ are unipotent.
Let $V$ be the representation space and let $d=\dim V\le5$.

Let
\[
D=\log\sigma(b),\qquad
M_i=\log\sigma(ab^{1-i})=M_0-iD\quad(i=0,1,2).
\]
The $M_i$ are commuting, conjugate nilpotents, and $D\ne0$.
They are nonzero, since otherwise conjugacy would give $M_0=M_1=0$
and hence $D=0$.
Let $X=\sigma(x)$ and $Y=\sigma(y)$. We have
\[
X^{-1}M_1X=M_0,\qquad Y^{-1}M_1Y=M_2.
\]
Thus a common image or kernel of the same positive power of all
three $M_i$ is invariant under $X,Y$, and also under
$\sigma(a),\sigma(b)$. By irreducibility, no such subspace is
proper and nonzero.
Consequently $M_0$ and $M_1$ have distinct images and distinct
kernels, since equality would force the corresponding space
for $M_2=2M_1-M_0$ to coincide as well.

Suppose first that $M_0^2=0$. Conjugacy gives $M_1^2=M_2^2=0$,
and commutation gives
\[
0=M_2^2=(2M_1-M_0)^2=-4M_0M_1.
\]
Hence
\[
\operatorname{Im}M_0+\operatorname{Im}M_1
\ \subseteq \ 
\ker M_0\cap\ker M_1.
\]
Let $r=\operatorname{rank}M_0$. The distinct images and kernels imply $r+1\le d-r-1$. Moreover $r\ne1$, since two
rank-1 maps with distinct images and kernels have a
rank-2 linear combination $2M_1-M_0$. Thus $d\ge2r+2\ge6$.

Now let $k\ge2$ be maximal with $M_0^k\ne0$.
Since $d\le5$, the largest Jordan block is unique, so each
$\operatorname{Im}M_i^k$ is a line.
If $M_0$ has at most two Jordan blocks, then $\dim\ker M_0\le2$.
The restriction of $M_1$ to $\ker M_0$ is nonzero and nilpotent,
since the kernels are distinct. Hence $\dim\ker M_0=2$ and
$\ker M_0\cap\ker M_1$ is a line.
Each $\operatorname{Im}M_i^k$ is a line invariant under the
commuting nilpotents $M_0,M_1$, which therefore vanish on it.
Thus all three images equal $\ker M_0\cap\ker M_1$, a contradiction.
The only remaining Jordan type is $(3,1,1)$.

Let $K=\ker M_0$ and $\ell=\operatorname{Im}M_0^2$.
Choose a Jordan chain $M_0v_3=v_2$, $M_0v_2=v_1$, so that
\[
V=\langle v_3,v_2\rangle\oplus K,\qquad \ell=\langle v_1\rangle.
\]
Commutation and nilpotence give
\[
Dv_3=\alpha v_2+w,\qquad Dv_2=\alpha v_1,\qquad Dv_1=0,
\qquad w\in K.
\]
If $DK\subseteq\ell$, then
$\operatorname{Im}(M_0-sD)^2\subseteq\ell$ for every $s$.
For $s=0,1,2$, this gives the forbidden common image $\ell$.
Otherwise choose $z\in K$ with $Dz\notin\ell$.
For $s\ne0$ and $1-s\alpha\ne0$, the vectors
\[
\begin{aligned}
(M_0-sD)v_2&=(1-s\alpha)v_1,\\
(M_0-sD)z&=-sDz,\\
(M_0-sD)v_3&=(1-s\alpha)v_2-sw
\end{aligned}
\]
lie respectively in $\ell-\{0\}$, $K-\ell$,
and $V- K$, so they are linearly independent.
But $M_1$ and $M_2$ both have rank 2, forcing
$1-\alpha=1-2\alpha=0$, which is impossible.
\end{proof}


\begin{thebibliography}{9}
\raggedright


\bibitem{ButtonMin}
J.~O. Button,
\emph{Minimal dimension faithful linear representations of
common finitely presented groups},
preprint (2016),
\href{https://arxiv.org/abs/1610.03712}{arXiv:1610.03712}.

\bibitem{Button}
J.~O. Button,
\emph{Free by cyclic groups and linear groups with restricted
unipotent elements},
Groups Complex. Cryptol. \textbf{9} (2017), no.~2, 137--149.
\href{https://doi.org/10.1515/gcc-2017-0009}
{doi:10.1515/gcc-2017-0009}.

\bibitem{DTentries}
S.~Douba and K.~Tsouvalas,
\emph{Matrix entries, unipotents, and linearity of amalgams},
preprint (2026),
\href{https://arxiv.org/abs/2603.23969}{arXiv:2603.23969}.

\bibitem{DTflags}
S.~Douba and K.~Tsouvalas,
\emph{On transversality in flag manifolds and linearity of amalgams},
preprint (2026),
\href{https://arxiv.org/abs/2607.28863}{arXiv:2607.28863}.

\bibitem{Gersten}
S.~M.~Gersten,
\emph{The automorphism group of a free group is not a
$\operatorname{CAT}(0)$ group},
Proc.\ Amer.\ Math.\ Soc.\ \textbf{121} (1994), no.~4, 999--1002.

\bibitem{HWW}
N.~Hoda, D.~T.~Wise, and D.~J.~Woodhouse,
\emph{Residually finite tubular groups},
Proc.\ Roy.\ Soc.\ Edinburgh Sect.~A \textbf{150} (2020),
2937--2951.
\href{https://doi.org/10.1017/prm.2019.52}{doi:10.1017/prm.2019.52}.

\bibitem{Malcev}
A.~I.~Mal'cev,
\emph{On isomorphic matrix representations of infinite groups},
Rec.\ Math.\ [Mat.\ Sbornik] N.S.\ \textbf{8(50)} (1940), no.~3,
405--422.

\bibitem{Shalen}
P.~B.~Shalen,
\emph{Linear representations of certain amalgamated products},
J.\ Pure Appl.\ Algebra \textbf{15} (1979), no.~2, 187--197.

\end{thebibliography}
\end{document}